\documentclass[12pt]{amsart}

\usepackage{amssymb,latexsym}
\usepackage{enumerate}
\usepackage[hyperfootnotes=false]{hyperref}
\usepackage{fullpage}
\usepackage{soul}

\usepackage{fixme}
\fxsetup{
    status=draft,
    author=,
    layout=margin,
    theme=color
}

\makeatletter \@namedef{subjclassname@2010}{%
  \textup{2010} Mathematics Subject Classification}
\makeatother

\newcounter{thm} \numberwithin{thm}{section}
\newtheorem{Theorem}[thm]{Theorem}

\newtheorem{Lemma}[thm]{Lemma}
\newtheorem{Corollary}[thm]{Corollary}

\newtheorem{Claim}[thm]{Claim}

\newcommand{\F}{\mathbb{F}}

\newcommand{\cA}[0]{\mathcal A}	
	
\newcommand{\cC}[0]{\mathcal C}

\newcommand{\cH}[0]{\mathcal H}

\newcommand{\cL}[0]{\mathcal L}	
	
\newcommand{\thistheoremname}{}
\newtheorem*{genericthm*}{\thistheoremname}
\newenvironment{namedthm*}[1]
  {\renewcommand{\thistheoremname}{#1}%
   \begin{genericthm*}}
  {\end{genericthm*}}
\usepackage{tikz}
\usetikzlibrary{decorations.pathreplacing}
\tikzset{mybrace/.style={decoration={brace,raise=1.8mm},decorate}}
\tikzset{mybracedown/.style={decoration={brace,mirror,raise=1.8mm},decorate}}
\usetikzlibrary{math}
\usetikzlibrary{patterns}
\usetikzlibrary{matrix,backgrounds}

\newcommand{\kb}[1]{{\color{blue} #1}}
\newcommand{\orn}[1]{{\color{red} #1}}

\author[K. Bhowmick]{Krishnendu Bhowmick} \address{Johann Radon Institute for Computational and Applied Mathematics\\
Linz, Austria}
\email{Krishnendu.Bhowmick@oeaw.ac.at}

\author[O. Roche-Newton]{Oliver Roche-Newton} \address{Institute for Algebra, JKU\\
Linz, Austria}
\email{o.rochenewton@gmail.com}

\newcommand\blfootnote[1]{%
  \begingroup
  \renewcommand\thefootnote{}\footnote{#1}%
  \addtocounter{footnote}{-1}%
  \endgroup
}

\date{}

\begin{document}

\baselineskip=17pt

\title{Counting arcs in $\mathbb F_q^2$}

\blfootnote{Keywords: arcs, hypergraph containers}

\begin{abstract}  An arc in $\mathbb F_q^2$ is a set $P \subset \mathbb F_q^2$ such that no three points of $P$ are collinear. We use the method of hypergraph containers to prove several counting results for arcs. Let $\mathcal A(q)$ denote the family of all arcs in $\mathbb F_q^2$. Our main result is the bound
\[
|\mathcal A(q)| \leq 2^{(1+o(1))q}.
\]
This matches, up to the factor hidden in the $o(1)$ notation, the trivial lower bound that comes from considering all subsets of an arc of size $q$. 

We also give upper bounds for the number of arcs of a fixed (large) size. Let $k=q^t$ for some $t >2/3$, and let $\mathcal A(q,k)$ denote the family of all arcs in $\mathbb F_q^2$ with cardinality $k$.
We prove that, for all $\gamma >0$
\[
|\mathcal A(q,k)| \leq \binom{(1+\gamma)q}{k}.
\]
This result improves a bound of Roche-Newton and Warren \cite{RNW}. A nearly matching lower bound
\[
|\mathcal A(q,k)| \geq \binom{q}{k}
\]
follows by considering all subsets of size $k$ of an arc of size $q$. 
\end{abstract}

\date{}
\maketitle

\section{Introduction} 
Over a century ago, Dudeney \cite{HD} asked how many points can be placed in a \(n \times n\) grid such that no three are collinear. In 1951 Erd\H{o}s (published by Roth \cite{Roth}) showed that when \(n\) is a prime number the set \(\{(i,i^2\) mod \(n)\) : \(0\leq i < n\}\) contains no collinear triple, whereas a simple upperbound of \(2n\) follows from the pigeonhole principle. Hall et al. \cite{Hall} subsequently improved Erd\H{o}s's lower bound, but despite receiving considerable attention, the problem remains open. In this paper we will consider a closely related problem in the finite field setting.

Let $\F_q$ be the finite field of order $q = p^r$ for some prime $p$. An arc in $\F_q^2$ is a subset of $\F_q^2$ with no three points collinear. Let $\mathcal A(q)$ denote the family of all arcs in $\mathbb F_q^2$. One of the main goals of this paper is to provide bounds for the cardinality of $\mathcal A(q)$. For context, observe that the set
\begin{equation} \label{largearc}
 C=\{(x,x^2): x \in \mathbb F_q\}
 \end{equation}
 is an arc of cardinality $q$. This is essentially the same as the Erd\H{o}s construction mentioned above. Since each subset of an arc is also an arc, it immediately follows that
 \begin{equation} \label{arcstriv}
 |\mathcal A(q)| \geq 2^q.
 \end{equation}
 We prove the following upper bound, which almost matches \eqref{arcstriv}.
 \begin{Theorem} \label{thm:main1}
Let $\delta >0$ and let $q$ be a prime power which is sufficiently large with respect to $\delta$. The set $\mathcal A(q)$ of all arcs in $\mathbb F_q^2$ satisfies the bound
 \[
  |\cA(q)| \leq 2^{q+2q^{\frac{4}{5}+\delta}} . 
 \]
 In particular, $| \mathcal A(q)| \leq 2^{(1+o(1))q}$.
 \end{Theorem}

 We also consider the set $\cA(q,k)$ of all arcs of a fixed size $k$, with a focus on the case when $k$ is large. The problem of bounding the size of $\cA(q,k)$ was the main focus of a paper of the second author and Warren \cite{RNW}, and this paper is sequel to \cite{RNW}.
 For the case when $0 \leq k < q^{1/2}$, this question was largely settled in \cite{RNW}, where the bounds
 \begin{equation} \label{smallk}
\binom{q^2}{k} e^{\frac{-Ck^3}{q}} \leq |\cA(q,k)| \leq \binom{q^2}{k} e^{\frac{-ck^3}{q}}
 \end{equation}
 were established. In \eqref{smallk}, $c$ and $C$ are absolute constants. We interpret \eqref{smallk} as a statement that, for small $k$, a random set of $k$ elements has a fairly high probability of being an arc, with this probability even tending towards $1$ for $k=o(q^{1/3})$. 
 
Write $k=q^t$. For $t >1/2$, a significant change of behaviour in terms of the size of $\cA(q,k)$ was observed in \cite{RNW}. A better upper bound than that of \eqref{smallk} was established, see \cite[Theorem 2]{RNW}. However, it was not clear whether or not the upper bound given in \cite{RNW} for this range was optimal.

In this paper, we give improved upper bounds for the size of $\cA(q,k)$ when $k$ is large. 

\begin{Theorem} \label{thm:main2}
Let $\delta>0$ and suppose that $q$ is sufficiently large (with respect to $\delta$). Let $k=q^t$ with $\frac{2}{3}+2\delta<t \leq 1$. Then
 \[
  |A(q,k)| \leq \binom{(1+o(1))q}{k}.
 \]
\end{Theorem}

A more precise version of the statement of Theorem \ref{thm:main2} which gives information about the value hidden in the $o(1)$ term is given as Theorem \ref{thm:main2again}. Theorem \ref{thm:main2} is close to optimal; the lower bound
\[
\cA(q,k) \geq \binom{q}{k}
\]
can be seen by considering all subsets of size $k$  of the set $C$ described in \eqref{largearc}. 


This paper follows a similar approach to that of \cite{RNW}. In particular, the main tool is the method of hypergraph containers. The theory of hypergraph containers was developed independently by Balogh, Morris and Samotij \cite{BMS} and Saxton and Thomason \cite{ST}. We defer the full statement of the container theorem we use until Section \ref{sec:cont}. Roughly speaking, it says that if a hypergraph has a reasonably good edge distribution, we can obtain strong information about where the independent sets of the hypergraph may be found. 

In comparison with the predecessor paper \cite{RNW}, there are two main new ideas which allow us to break new ground.
\begin{itemize}
    \item To prove Theorem \ref{thm:main1}, we need a supersaturation lemma, namely Corollary \ref{cor:super4}, which gives good bounds for the number of arcs determined by sets with slightly more than $q$ elements, and in particular sets of size $q+x$ when $x=o(q)$. See Section \ref{sec:super} for more background on supersaturation results and their interaction with the method of hypergraph containers.
    \item As was the case in \cite{RNW}, we repeatedly apply the hypergraph container theorem to obtain a set of containers for the family of arcs $\cA(q)$. However, we observe that, as these containers get smaller, we can eventually make use of a non-trivial bound for the maximum co-degree $\Delta_2$. See Section \ref{sec:cont} for the definition of this parameter. This results in better quantitative information about the set of containers, which in turn yields the improved bound stated in Theorem \ref{thm:main2}.
\end{itemize}



\section{Containers and supersaturation}

\subsection{Statement of the container theorem} \label{sec:cont}

The main tool of this paper is a container theorem for $3$-uniform hypergraphs.  The method of hypergraph containers has had a remarkable impact on extremal combinatorics in recent years (see for example \cite{BalSam} and \cite{MS}). This impact is also being felt in Additive Combinatorics (see \cite{BLS} and \cite{PO}) and discrete geometry (see \cite{BS}). See \cite{BMS2} for a fairly recent survey of this topic.



Before stating the container theorem to be used, it is necessary to introduce some related quantities. Since we will only apply the container theorem for $3$-uniform hypergraphs, we give all of the definitions we need only for this case. A more general form of the statement and definitions we need, adapted to $k$-uniform hypergraphs, can be found in \cite{BS}, amongst other places. 

For a $3-$uniform hypergraph $\cH=(V,E)$ and $v \in V$, $d(v)$ denotes the degree of $v$, i.e. the number of edges which contain $v$. Let $d(\cH)$ denote the average degree of $\cH$, so
\begin{equation} \label{ddefn}
d(\cH)= \frac{1}{|V|}\sum_{v \in V} d(v)=  \frac{3|E|}{|V|}.
\end{equation}
We can also define the \emph{co-degree} for a subset $S \subseteq V$ of vertices as
$$d(S) = \{ e \in E(\cH) : S \subseteq e \}.$$
Using this definition we define the \emph{maximum co-degree} $\Delta_2(\mathcal H)$ as
$$\Delta_2(\cH) = \max_{\substack{S \subseteq V \\ |S| = 2}}d(S).$$
More generally, one can define the parameter $\Delta_t(\cH)$, which counts the maximum co-degree among all sets of cardinality $t$. However, since we only consider $3$-uniform hypergraphs in this paper, it turns out that we only use this definition for the case $t=2$.

For any $V' \subset V$, $\cH[V']$ denotes the subgraph induced by $V'$.

We now state the container theorem we need, which is a special case of Corollary 3.6 in \cite{ST}.

\begin{Theorem} \label{container}
Let $\cH=(V,E)$ be a $3-$uniform hypergraph on $n$ vertices, and let $\epsilon, \tau \in (0,1/2)$. Suppose that
\begin{equation} \label{cond1} 
\frac{\Delta_2(\cH)}{d(\cH) \cdot \tau}+ \frac{1}{2d(\cH) \cdot \tau^{2}} \leq \frac{\epsilon}{288}
\end{equation}
and
\begin{equation} \label{cond2}
\tau < \frac{1}{3600}.
\end{equation}
Then there exists a set $\cC$ of subsets of $V$ such that
\begin{enumerate}
    \item if $A \subset V$ is an independent set then there exists $C \in \cC$ such that $A \subset C$;
    \item $|E(\cH[C])| \leq \epsilon |E(\cH)|$ for all $C \in \cC$;
    \item $\log |\cC| \leq c n  \tau \log(\frac{1}{\epsilon})  \log(\frac{1}{\tau})$,  
\end{enumerate}
where $c$ is an absolute constant (we can take $c= 108000$).
\end{Theorem}
The set $\cC$ above is referred to as the set of \emph{containers,} and a set $C \in \cC$ is itself a container.

\subsection{Supersaturation results} \label{sec:super}

In all applications of the method of hypergraph containers, it is necessary to have a \textit{supersaturation} result. In general terms, a supersaturation lemma is a result that says that, once we have enough elements in our set to guarantee the existence of a certain substructure, we quickly generate many copies of the substructure in question. In our case, this means that we need to show that sets in $\mathbb F_q^2$ with significantly more than $q$ elements must contain many collinear triples.

Given a set $P \subset \mathbb F_q^2$, let $T(P)$ denote the number of collinear triples in $P$. That is,
\[
T(P)= |\{S \subset P : |S|=3 \text{ and the three elements of $S$ are collinear}\}|.
\]

From a result of Segre \cite{Segre}, it follows that any point-set of size \(q+2\) in \(\mathbb{F}_q^2\) contains a collinear triple, making the construction \eqref{largearc} almost best possible. In \cite{RNW}, an application of the Cauchy-Schwarz inequality was used to prove that, for an unspecified absolute constant $c$,
\begin{equation} \label{super1}
|P| \geq 4q \Rightarrow T(P) \geq c \frac{|P|^3}{q}.
\end{equation}
This statement is optimal up to the multiplicative constant. This can be seen by taking a $p$-random subset of $\mathbb F_q^2$ for some $p$ with order of growth at least $1/q$. However, \eqref{super1} does not say anything about what happens in the range when $q<|P|<4q$, i.e. when $P$ is only slightly larger than the threshold for guaranteeing at least one arc. In order to prove Theorem \ref{thm:main1}, a good supersaturation result for this range is required. We prove the following rather general supersaturation result for arcs, which includes information about this range.

\begin{Lemma}\label{lem:super3}
Let $k \in \mathbb N $ and let $x$ be an integer satisfying $0 \leq x \leq q+1$. Let $P \subset \F_q^2$ with $|P| =  k(q+1)+x+1$. Then 
\[
T(P) \geq \frac{1}{3}\left(\binom{k}{2}(q+1) + k\cdot x\right)\cdot |P|.
\]
\end{Lemma}

%


Recall that existence of a collinear triple in a set of size \(q+2\) of \(\mathbb{F}_q^2\) follows from the result of Segre \cite{Segre}. From Lemma \ref{lem:super3} we can say a set of  size \(q+2+x\) of \(\mathbb{F}_q^2\) for \(x \in \mathbb{N}\) will contain more than \(\frac{qx}{3}\) collinear triples. 

Before proving Lemma \ref{lem:super3}, it is convenient to record two corollaries that will be used later. 

\begin{Corollary} \label{cor:super4}
Let $q$ be a prime power and let $\frac{4}{q}<\gamma \leq 1$. Suppose $P \subset \mathbb F_q^2$ such that 
\[
T(P) \leq \frac{\gamma q^2}{6}.
\]
Then $|P| \leq (1+\gamma)q$.
\end{Corollary}

\begin{proof} 
We will prove the contrapositive form of the statement; assuming that $|P| >(1+ \gamma)q$, we prove that $T(P) > \frac{\gamma q^2}{6}$.

Since $|P| > (1+ \gamma)q \geq q+2 $, Lemma \ref{lem:super3} can be applied with $k=1$. We have
\[
|P|=q+2+x > (1+ \gamma) q,
\]
and thus $x > \gamma q - 2 \geq \frac{\gamma q}{2}$. The latter inequality follows from the assumption that $ \gamma \geq \frac{4}{q}$. Lemma \ref{lem:super3} then gives
\[
T(P) \geq \frac{1}{3}x|P| > \frac{1}{3} \frac{\gamma q}{2} |P| \geq \frac{ \gamma q^2}{6},
\]
as required.
\end{proof}

Lemma \ref{lem:super3} can be used to reprove \eqref{super1}. We state the following version, with a concrete multiplicative constant.

\begin{Corollary}\label{cor:super5}
Let $q \geq 8$ be a prime power and let $P \subset \F_q^2$ with $|P| \geq  2q$. Then
\[
T(P) \geq \frac{|P|^3}{64q}.
\]
\end{Corollary}

\begin{proof}
Write $|P|=k(q+1)+x+1$ with $0 \leq x \leq q+1$. By Lemma \ref{lem:super3},
\begin{align*}
    T(P) & \geq \frac{1}{3}\left(\binom{k}{2}(q+1) + k\cdot x\right)\cdot |P|
    \\& \geq  \frac{k}{6}\left((k-1)(q+1) +  x\right)\cdot |P|
    \\& = \frac{k}{6}\left(|P|-q-2 \right)\cdot |P|
    \\& = \frac{1}{6} \cdot  \frac{|P|-x-1}{q+1} \cdot \left(|P|-q-2 \right)\cdot |P|
    \\& \geq \frac{1}{6} \cdot  \frac{|P|-q-2}{q+1} \cdot \left(|P|-q-2 \right)\cdot |P|
    \\& \geq \frac{1}{6} \cdot  \frac{|P|}{4q} \cdot \left(\frac{3|P|}{8} \right)\cdot |P| = \frac{|P|^3}{64q}.
\end{align*}
The final inequality uses the assumption that $q \geq 8$.
\end{proof}

We now turn towards the proof of Lemma \ref{lem:super3}. The proof uses Karamata's inequality. For two finite non-increasing sequences of real numbers $x_1,\dots x_n$ and $y_1,\dots y_n$, we say that $x_1,\dots ,x_n$ \textit{majorizes} $y_1,\dots, y_n$ if
\begin{equation} \label{major1}
    x_1+\dots + x_i \geq y_1+ \dots + y_i, \,\,\,\,\,\,\, \forall \,\,\, 1 \leq i \leq n
\end{equation}
and
\begin{equation} \label{major2}
    x_1+\dots +x_n = y_1+ \dots +y_n.
\end{equation}

\begin{Lemma}[Karamata's Inequalty] \label{lem:karamata}
Let $f:I \rightarrow \mathbb R$ be a convex function defined on an interval $I$. Suppose that $(x_i)_{i =1,\dots,n}$ and $(y_i)_{i = 1, \dots n}$ are non-increasing sequences in $I$ such that $(x_i)_{i =1,\dots,n}$ majorizes $(y_i)_{i =1,\dots,n}$. Then 
\[
\sum_{i=1}^{n}f(x_i) \geq \sum_{i=1}^{n}f(y_i).
\]
\end{Lemma}

\begin{proof}[Proof of Lemma \ref{lem:super3}]

For any point \(v\) in \(\F_q^2\); there are exactly \((q+1)\) lines passing through it. Let the set of lines passing through \(v\) be \(\mathcal{L}(v)\). We label the lines in $\cL(v)$ according the number of elements of $P$ they contain, in non-increasing order. That is, we write
\begin{equation*}
\mathcal{L}(v) = \{\ell_1^{(v)}, \dots, \ell_{q+1}^{(v)} \},
\end{equation*}
such that 
\[
i < j \Rightarrow |\ell_i^{(v)} \cap P| \geq |\ell_j^{(v)} \cap P|.
\]
We assign weights to points in \(P\) in a way such that the sum of weight of all the points in \(P\) will be $T(P)$.  For a point \(v\) in \(P\) assign a weight \(W(v)\) to it as follows:
\begin{equation}\label{eq:W_i}
W(v) := \frac{1}{3} \sum_{j=1}^{q+1} \binom{w_j^{(v)}}{2};
\end{equation}
where \(w_j^{(v)} = |\ell_j^{(v)} \cap P|- 1\) for all \(j \in [q+1]\). Note that $W(v)$ counts one third of the number of collinear triples in $P$ which contain $v$. Therefore, assigning weights in this fashion satisfies the required property; that is, 
\begin{equation}\label{eq:P}
    T(P) = \sum_{v\in P} W(v).
\end{equation}
Also observe that for any \(v\in P\),
\begin{equation}\label{eq:sum w_{i,j}}
    \sum_{j=1}^{q+1} w_{j}^{(v)} = |P|-1= k(q+1) + x.
\end{equation}
We will state and prove the following claim in order to finish the proof. 
\begin{Claim}\label{claim:karamata}
For any point \(v\) in \(P\),
\[
W(v) \geq \frac{1}{3}\left(\binom{k}{2}(q+1) + k\cdot x\right).
\]
\end{Claim}
\begin{proof}[Proof of Claim]
We observed in (\ref{eq:sum w_{i,j}}) that the sum $\sum_{j=1}^{q+1} w_{j}^{(v)}$ is fixed, for all $v \in P$. It's a natural intuition from (\ref{eq:W_i}) that the quantity \(W(v)\) reaches its minimum when the \(w_{j}^{(v)}\) terms are almost equal, in other words \(|w_{j}^{(v)}-w_{j'}^{(v)}| \leq 1\) for all \(j,j'\in [q+1]\). We will prove that this intuition is indeed true via an application of Karamata's inequality. 

The function \(f(y) = \binom{y}{2}\) is convex, and the non-increasing sequence consisting of \(x\) occurrences of \((k+1)\) and \(q-x+1\) occurrences of \(k\) is majorized by every other non-increasing sequence of integers of length \(q+1\) and total sum \(k(q+1)+x\). It therefore follows from Karamata's Inequality that
\begin{equation*}
  W(v) = \frac{1}{3} \sum_{j=1}^{q+1} \binom{w_{j}^{(i)}}{2} \geq \frac{1}{3} \left( \sum_{j=1}^{x} \binom{k+1}{2} + \sum_{j=x+1}^{q+1} \binom{k}{2}\right) = \frac{1}{3}\left(\binom{k}{2}(q+1) + k\cdot x\right).
\end{equation*}
This finishes the proof of the Claim \ref{claim:karamata}.
\end{proof}
Alternatively, instead of using Karamata's inequality Claim \ref{claim:karamata} can also be proved by using the fact that
\[
\binom{y}{2} +\binom{z}{2} \geq \binom{y+i}{2} +\binom{z-i}{2}
\]
for all \(z \geq z-i \geq y+i \geq y \geq 0\).

Now applying Claim \ref{claim:karamata} in (\ref{eq:P}) we conclude that
\begin{equation*}
T(P)= \sum_{v\in P} W(v) \geq   \sum_{v\in P}\frac{1}{3}\left(\binom{k}{2}(q+1) + k\cdot x\right) = \frac{1}{3}\left(\binom{k}{2}(q+1) + k\cdot x\right)\cdot |P|.
\end{equation*}
Hence Lemma \ref{lem:super3} is proved.

\end{proof}

\section{Container lemmas for arcs}

\subsection{Basic properties of the graph encoding triples}
Define a $3-$uniform hypergraph $\cH$ with vertices corresponding to points in $\F_q^2$, with three points forming a hyperedge if they are collinear. Note that the number of edges in this graph is
\[
q(q+1) \binom{q}{3} <q^5.
\]
In this section, we will make repeated applications of the container theorem for this graph and its induced subgraphs until we obtain a family of containers for arcs in $\mathbb F_q^2$ (i.e. a family $\cC$ of subsets of $\mathbb F_q^2$ with the property that, for any arc $P \in \mathbb F_q^2$, there exists $C \in \cC$ such that $P \subset C$) with the properties we need. Before starting this iterative process, we collect a few inequalities that will be used repeatedly in the proofs of the forthcoming three lemmas.

A reformulation of Corollary \ref{cor:super5} states that, for any $A \subset \mathbb F_q^2$,
\begin{equation} \label{Acases}
|A| \leq \max \{2q, 4q^{1/3} T(A)^{1/3} \}.
\end{equation}
It therefore follows from the definition \eqref{ddefn} that
\begin{equation} \label{dbound}
    T(A) \geq q^2 \Rightarrow d( \cH[A]) \geq \frac{3T(A)^{2/3}}{4q^{1/3}} .
\end{equation}
We will need to bound the quantity $\Delta_2(\cH[A])$ to apply Theorem \ref{container}.
A first observation is that, for any $A \subset \mathbb F_q^2$,
\begin{equation} \label{Delta21}
    \Delta_2(\cH[A]) \leq q-2.
\end{equation}
Indeed, given a pair of points in the plane, the number of points in $\mathbb F_q^2$ which are collinear with the given pair is $q-2$, and \eqref{Delta21} follows.

 A better bound for $\Delta_2(\cH[A])$ is available when $T(A)$ is smaller. We have
 \begin{equation} \label{Delta22}
\Delta_2(\cH[A]) \leq 2(T(A))^{1/3}.
\end{equation}
Indeed, 
\[
\Delta_2(\cH[A])= \max_{\ell \in \{ \text{all lines in } \mathbb F_q^2 \}} | \ell \cap A| -2:=M-2.
\]
A line containing $M$ elements of $P$ gives rise to $\binom{M}{3}$ collinear triples, and so
\[
T(A) \geq \binom{M}{3} \geq \frac{(\Delta_2( \cH[A]))^3}{6}.
\]
A rearrangement of this inequality gives \eqref{Delta22}.

\subsection{A first container lemma for arcs}

The argument for our first container lemma largely follows that of \cite[Lemma 2]{RNW}. The only difference in what follows is that we keep track of the number of collinear triples determined by the containers, rather than their size.

\begin{Lemma} \label{lem:cont1}
Let $\delta>0$ and suppose that $q$ is a sufficiently large (with respect to $\delta$) prime power. Then there exists a family $\mathcal C^1$ of subsets of $\mathbb F_q^2$ such that
\begin{itemize}
    \item $|\mathcal C^1| \leq 2^{q^{2/3+\delta}}$,
    \item For all $C \in \mathcal C^1$, $T(C) \leq q^{3}$,
    \item For every arc $P \subset \mathbb F_q^2$, there exists $C \in \mathcal C^1$ such that $P \subset C$.
\end{itemize}
\end{Lemma}

\begin{proof}

We employ an idea used in \cite{BS}; we will iteratively apply Theorem \ref{container} to subsets of $\mathbb F_q^2$. We begin by applying it to the graph $\cH$ encoding collinear triples, which we defined at the beginning of this section. Note that independent sets in this hypergraph are the same thing as arcs in $\mathbb F_q^2$. As a result, we obtain a set $\mathcal C_{\mathbb F_q^2}$ of containers. We iterate by considering each $A \in \mathcal C_{\mathbb F_q^2}$. If $A$ contains too many collinear triples, then we apply Theorem \ref{container} to the graph $\cH[A]$ to get a family of containers $\mathcal C_A$. If the number of collinear triples in $A$ is sufficiently small then we put this $A$ into a final set $\mathcal C$ of containers (or to put it another way, we write $\mathcal C_A=A$). 

Repeating this for all $A \in \mathcal C_1$ we obtain a new set of containers
\[
\mathcal C_2 = \bigcup_{A \in  \mathcal C_1} \mathcal C_A.
\]
Note that $\mathcal C_2$ is a container set for $\cH$. Indeed, suppose that $X$ is an independent set in $\cH$. Then there is some $A \in \mathcal C_1$ such that $X \subset A$. Also, $X$ is an independent set in the hypergraph $\cH[A]$, which implies that $X \subset A'$ for some $A' \in \mathcal C_A \subset \mathcal C_2$.

We then repeat this process, defining
\[
\mathcal C_i = \bigcup_{A \in  \mathcal C_{i-1}} \mathcal C_A.
\]
By choosing the values of $\tau$ and $\epsilon$ appropriately, we can ensure that after relatively few steps we have $T(A) \leq q^3$ for all of the sets $A \in \mathcal C_m$. We then declare $\mathcal C^1= \mathcal C_m$. It turns out that, because of $m$ being reasonably small, $|\mathcal C^1|$ is also fairly small.

\begin{figure}[ht]

\centering

\begin{tikzpicture}[scale =1.2]
       \draw  (0,0) -- (-3,-1);
       \draw (0,0) -- (-1.5,-1);
        \node[above] at (0.2,-0.8) {$\dots$};
         \draw (0,0) -- (2,-1);

\node[above, scale=0.85] at (0,0) {$\mathbb F_q^2$};

\draw[fill] (0,0) circle [radius=0.06];

\draw[fill] (-1.5,-1) circle [radius=0.06];
\draw[fill] (-3,-1) circle [radius=0.06];
\draw[fill] (2,-1) circle [radius=0.06];

\node[above, scale=0.85] at (-3,-1) {$A_1 \in \cC_{\mathbb F_q^2}$};
\node[below, scale=0.85] at (-1.5,-1) {$A_2 \in \cC_{\mathbb F_q^2}$};
\node[above, scale=0.85] at (2,-1) {$A_k \in \cC_{\mathbb F_q^2}$};

  \draw  (-4,-2) -- (-3,-1);
  \draw  (-2,-2) -- (-3,-1);
         \node[above] at (-3,-1.7) {$\dots$};
         
         \draw  (1,-2) -- (2,-1);
  \draw  (3,-2) -- (2,-1);
         \node[above] at (2,-1.7) {$\dots$};
         
         \draw[fill] (3,-2) circle [radius=0.06];
\draw[fill] (1,-2) circle [radius=0.06];
\draw[fill] (-2,-2) circle [radius=0.06];
\draw[fill] (-4,-2) circle [radius=0.06];

\node[below, scale=0.85] at (-4,-2) {$A_{1,1} \in \cC_{A_1}$};
\node[below, scale=0.85] at (-2,-2) {$A_{1,k_1} \in \cC_{A_1}$};

  \node[above] at (0,-2.7) {$\vdots$};
    \node[above] at (-3,-2.7) {$\vdots$};  
    \node[above] at (2,-2.7) {$\vdots$};







\end{tikzpicture}

\caption{An illustration of the beginning of the construction of the set of containers $\cC^1$, which can be viewed as a tree. For each $A_i \in \cC_{\mathbb F_q^2}$, we check the size of $T(A_i)$, and if it is still too large we apply the container theorem again. If the set contains few enough collinear triples, we stop the process. This is what happens with the set $A_2$ in the diagram. The final set of containers $\cC^1$ consists of the leaves of the tree.}

\end{figure}
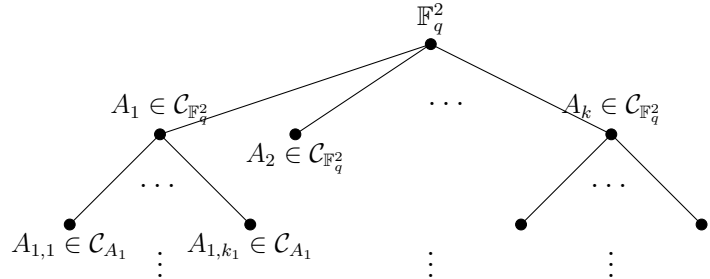

Now we give more precise details of how to run this argument. Let $A \in \mathcal C_j$, with $j \leq m$, and write $T(A)=q^{5-s}$. If $s \geq 2$ then we add $A$ to the final set of containers $\cC^1$. Otherwise, we apply Theorem \ref{container} to $\cH[A]$. We choose the parameters $\epsilon$ and $\tau$ to be
\begin{equation} \label{epsdel}
\epsilon= q^{- \delta}, \,\,\,\,\,\,\,\, \tau= \max \{ 1000q^{\frac{2s}{3}+\delta-2}, 100 q^{\frac{s}{3}+\frac{\delta}{2}-\frac{3}{2} }\}.
\end{equation}

In order for this application of Theorem \ref{container} to be legitimate, we need to make some calculations involving $d(H[A])$ and $\Delta_2(H[A])$, and check that the conditions of Theorem \ref{container} are satisfied. The main challenge is to verify that \eqref{cond1} holds.

Since we have $T(A) \geq q^3$, it follows from \eqref{dbound} that
\begin{equation} \label{dbound2}
d(\cH[A]) \geq  \frac{3T(A)^{2/3}}{4q^{1/3}}=\frac{3}{4}q^{3-\frac{2s}{3}}.
\end{equation}
Applying \eqref{dbound2} and \eqref{Delta21} and then using the two components of the definition of $\tau$, it follows that
\begin{align*} 
\frac{\Delta_2(\cH)}{d(\cH) \cdot \tau}+ \frac{1}{2d(\cH) \cdot \tau^{2}} &\leq \frac{4}{3q^{2-\frac{2s}{3}}\cdot \tau}+ \frac{1}{q^{3-\frac{2s}{3}} \cdot \tau^{2}}
\\& \leq \frac{1}{500q^{\delta}} + \frac{1}{10000q^{\delta}}
\\& \leq \frac{\epsilon}{288}.
\end{align*}
We have therefore verified that \eqref{cond1} holds. Since $s \leq 2$, the condition \eqref{cond2} is guaranteed to hold as long as we choose $q$ to be sufficiently large with respect to $\delta$. The condition $\epsilon <1/2$ follows similarly.

Theorem \ref{container} therefore gives the bound
\begin{equation} \label{contbound}
|\cC_A| \leq 2^{c  \delta |A|\tau (\log q)^2},
\end{equation}
for some absolute constant $c$. Applying \eqref{Acases} and the fact that $T(A)=q^{5-s} \geq q^3 \geq q^2$, we see that
\begin{equation} \label{Abound}
|A| \leq 4q^{2-s/3}.
\end{equation}
It then follows from \eqref{Abound}, \eqref{contbound}, the definition of $\tau$, and the upper bound $s \leq 2$, that
\[
|\cC_A| \leq 2^{c'\delta q^{\frac{2}{3}+ \delta} (\log q)^2},
\]
for some absolute constant $c'$.

Since each step of this process reduces the number of edges in the containers by a factor of $q^{\delta}$, it follows that the process will terminate after at most $2/\delta$ steps. The final set $\cC^1$ therefore contains at most
\[
\left(2^{c'\delta q^{\frac{2}{3}+ \delta} (\log q)^2} \right)^{\frac{2}{\delta}} = 2^{2c'q^{\frac{2}{3}+ \delta} (\log q)^2}
\]
elements. By choosing $q$ to be sufficiently large, we can absorb the constant and logarithmic terms into a slightly larger exponent and conclude that
\[
|\cC^1| \leq 2^{q^{\frac{2}{3}+2\delta}}.
\]

\end{proof}

\subsection{A second container lemma for arcs}

We use Lemma \ref{lem:cont1} as a basis for the following improved result, in which we reduce the number of edges further without paying too serious a price. The main quantitative cost of this reduction is an increase in the size of the final set of containers, and this increase depends on the new error-term parameter $\gamma$.

\begin{Lemma} \label{lem:cont2}
Let $\delta>0$ and suppose that $q$ is a sufficiently large (with respect to $\delta$) prime power. Fix an absolute constant $c$ and let $cq^{-1/2+ 3\delta/2} \leq \gamma \leq 1$. Then there exists a family $\mathcal C^2$ of subsets of $\mathbb F_q^2$ such that
\begin{itemize}
    \item $|\mathcal C^2| \leq 2^{\gamma^{-\frac{2}{3}}q^{\frac{2}{3}+2\delta}}$, 
    \item For all $C \in \mathcal C^2$, $|C| \leq (1+ \gamma)q$,
    \item For every arc $P \subset \mathbb F_q^2$, there exists $C \in \mathcal C^2$ such that $P \subset C$.
\end{itemize}
\end{Lemma}

\begin{proof}

Let $\cC^1$ be the set of containers given by Lemma \ref{lem:cont1}, and fix $C \in \cC^1$. Similar to the proof of Lemma \ref{lem:cont1}, we will iteratively apply Theorem \ref{container} to obtain a family $\cC^2$ of containers whose components contain fewer collinear triples. 

Let $A $ be a container obtained during this iterative process and write $T(A)=\frac{\gamma}{6}q^{5-s}$. If $s \geq 3$ then we add $A$ to the final set of containers $\cC^2$. Otherwise, we will apply Theorem \ref{container} to $\cH[A]$ to obtain a family $\cC_A$ of containers for $\cH[A]$.

The main difference between this proof and that of Lemma \ref{lem:cont1} is that we  bound $\Delta_2(\cH[A])$ using \eqref{Delta22}. Since $T(A) \leq q^3$, this bound is better than the bound \eqref{Delta21} that we used in the proof of Lemma \ref{lem:cont1}. In particular, we have
\begin{equation} \label{Delta221}
\Delta_2(\cH[A]) \leq 2(T(A))^{1/3}  \leq \gamma^{1/3}2q^{\frac{5}{3}-\frac{s}{3}}. 
\end{equation}
Using \eqref{Acases}, along with the fact that $T(A) \leq \gamma q^{5-s}$ and $s \leq 3$, yields
\begin{equation} \label{Aboundagain}
|A| \leq\max \{2q, 4q^{1/3} T(A)^{1/3} \} \leq  \max \{2q, 4\gamma^{\frac{1}{3}}q^{2- \frac{s}{3}} \} \leq 4q^{2-\frac{s}{3}}.
\end{equation}
It therefore follows that
\begin{equation} \label{dboundagain}
d( \cH[A]) = \frac{3T(A)}{|A|} \geq \frac{\frac{\gamma}{2}q^{5-s}}{4q^{2-\frac{s}{3}}} =  \frac{\gamma}{8} q^{3-\frac{2s}{3}}.
\end{equation}

We choose the parameters $\epsilon$ and $\tau$ to be
\begin{equation} \label{epsdel2}
\epsilon= q^{- \delta}, \,\,\,\,\,\,\,\, \tau=8000\gamma^{-\frac{2}{3}}q^{\frac{s}{3}+\delta-\frac{4}{3}}.
\end{equation}
We again need to check that \eqref{cond1} holds. Indeed, by \eqref{Delta221} and \eqref{dboundagain},
\begin{equation*} 
\frac{\Delta_2(\cH)}{d(\cH) \cdot \tau}+ \frac{1}{2d(\cH) \cdot \tau^{2}} \leq \frac{q^{-\delta}}{500}+ \frac{\gamma^{1/3} q^{-\frac{1}{3}-2\delta}}{10000}
 \leq \frac{q^{-\delta}}{500}+ \frac{q^{-\frac{1}{3}-2\delta}}{10000} 
 \leq \frac{\epsilon}{288}.
\end{equation*}
We have therefore verified that \eqref{cond1} holds. The condition $0<\epsilon < 1/2$ is guaranteed by choosing $q$ to be sufficiently large with respect to $\delta$. We also need to ensure that \eqref{cond2} holds. This follows from the assumption that $\gamma >cq^{-1/2+ 3\delta/2}$, along with the fact that $s \leq 3$, provided that the absolute constant $c$ is chosen to be sufficiently large. Indeed,
\[
\tau=8000\gamma^{-\frac{2}{3}}q^{\frac{s}{3}+\delta-\frac{4}{3}} \leq \frac{8000}{c}q^{\frac{s}{3}-1} \leq \frac{8000}{c} \leq \frac{1}{3600}.
\]
Theorem \ref{container} can therefore be legitimately applied, and it gives the bound
\begin{equation} \label{contprelim}
|\cC_A| \leq 2^{c'\delta|A|\tau (\log q)^2},
\end{equation}
for some absolute constant $c'$. Applying \eqref{Aboundagain} and the definition of $\tau$, one obtains the bound
\[
\tau |A| \leq C \max  \left \{ q^{-\frac{1}{3}+\frac{s}{3}+\delta}\gamma^{-\frac{2}{3}} , q^{\frac{2}{3}+\delta}\gamma^{-1/3} \right \} \leq C q^{\frac{2}{3}+\delta}\gamma^{-\frac{2}{3}},
\]
where $C$ is an absolute constant. Combining this with \eqref{contprelim} gives
\[
|\cC_A| \leq 2^{c''\delta\gamma^{-\frac{2}{3}} q^{\frac{2}{3}+ \delta} (\log q)^2},
\]
with a new absolute constant $c''$.

Since each step of this process reduces the number of collinear triples in the containers by a factor of $q^{\delta}$, it follows that the process will comfortably terminate after at most $2/\delta$ steps. We also need to take into account that we already started with $|\cC^1|$ containers, and that we apply this process for each of the elements of $\cC^1$. Therefore, the final set $\cC^2$ contains at most
\[
|\cC^1|\left(2^{c''\delta\gamma^{-\frac{2}{3}} q^{\frac{2}{3}+ \delta} (\log q)^2}\right)^{\frac{2}{\delta}} =|\cC^1| 2^{2c''\gamma^{-\frac{2}{3}}q^{\frac{2}{3}+ \delta} (\log q)^2} \leq 2^{q^{\frac{2}{3}+\delta}}2^{2c''\gamma^{-\frac{2}{3}}q^{\frac{2}{3}+ \delta} (\log q)^2} \leq 2^{3c''\gamma^{-\frac{2}{3}}q^{\frac{2}{3}+ \delta} (\log q)^2}  
\]
elements. By choosing $q$ to be sufficiently large with respect to $\delta$, we can absorb the constant and logarithmic terms into a slightly larger exponent and conclude that
\[
|\cC^2| \leq 2^{\gamma^{-\frac{2}{3}}q^{\frac{2}{3}+2\delta}}.
\]

It remains to check that $|A| \leq q(1+\gamma)$ for all $A \in \cC$. This follows immediately from Corollary \ref{cor:super4}.

\end{proof}

\section{Counting arcs}

We are now ready to prove the main results of the paper. We start with Theorem \ref{thm:main1}, which is restated below for convenience.

\begin{Theorem}
Let $\delta>0$. Then, for all $q$ sufficiently large with respect to $\delta$,
 \[
  |\cA(q)| \leq 2^{q+2q^{\frac{4}{5}+2\delta}}.
 \]
\end{Theorem}

\begin{proof}
 Let $\cC^2$ denote the set of containers given by Lemma \ref{lem:cont2}, applied with $\gamma=q^{-\frac{1}{5}+\frac{6\delta}{5}}$. All of the sets in $\mathcal A(q)$ are subsets of some $C \in \cC^2$. Since $|C| \leq (1+ \gamma)q =  q + q^{\frac{4}{5}+ \frac{6\delta}{5}}$, it follows that
 \[
  |\mathcal A(q)| \leq |\cC^2| 2^{q+q^{\frac{4}{5}+\frac{6\delta}{5}}} \leq 2^{q^{\frac{4}{5}+\frac{6\delta}{5}}} \cdot2^{q+q^{\frac{4}{5}+\frac{6\delta}{5}}} \leq 2^{q+2q^{\frac{4}{5}+2\delta}}.
 \]

\end{proof}

 Now we prove Theorem \ref{thm:main2}, which is restated below in an equivalent form for convenience.
 
  \begin{Theorem} \label{thm:main2again}
Let $\delta> 0$ and suppose that $q$ is sufficiently large with respect to $\delta$. Let $k=q^t$ with $\frac{2}{3}+2\delta<t \leq 1$. Define $\gamma=q^{\frac{2}{5}-\frac{3t}{5} +\delta}$. Then
 \[
  |\cA(q,k)| \leq \binom{(1+2\gamma)q}{k}.
 \]
\end{Theorem}

\begin{proof}
Apply Lemma \ref{lem:cont2} with this choice of $\delta$ and $\gamma$. The assumption that $\frac{2}{3}+2\delta<t \leq 1$ is sufficient to ensure that the condition on $\gamma$ in Lemma \ref{lem:cont2} is satisfied, provided that $q$ is sufficiently large with respect to $\delta$.  All of the sets in $\mathcal A(q,k)$ are subsets of some $C \in \cC^2$. Since $|C| \leq (1+\gamma)q$, it follows that 
 \begin{align*}
  |\mathcal A(q,k)| &\leq |\cC^2| \binom{ (1+\gamma)q}{k} 
  \\& \leq 2^{\gamma^{-\frac{2}{3}}q^{\frac{2}{3}+\delta}}\cdot  \binom{ (1+\gamma)q}{k} 
  \\&= 2^{q^{\frac{2}{5}+\frac{2t}{5}+\frac{\delta}{3}}}\cdot  \binom{ (1+\gamma)q}{q^t}
  \\& \leq \binom{(1+2\gamma)q}{q^t}.
  \end{align*} 
The final inequality above requires $q$ to be sufficiently large with respect to $\delta$.
\end{proof}

Finally, we consider the size of the largest arc contained in a random point set, improving a result from \cite{RNW}. Given $P \subset \mathbb F_q^2$, let $a(P)$ denote the size of the largest arc $P'$ such that $P' \subseteq P$. Let $Q_p$ be a random subset of $\mathbb{F}_{q}^{2}$ with the events $x \in Q_p$ being independent with probability $\mathbb{P}(x \in Q_p) = p$. We say that $Q_p$ is a \textit{$p$-random} set. 

\begin{Theorem} \label{thm:random}
Suppose that $p=q^{a}$ for some $-1/3<a<0$ and let $Q_p \subseteq \mathbb F_q^2$ be a $p$-random set. Let $f: \mathbb R \rightarrow \mathbb R$ be any function such that $\lim_{q \rightarrow \infty}f(q)= \infty$. Then
    \[
    \lim_{q \rightarrow \infty} \mathbb P [a(Q_p) \geq  qpf(q)]= 0.
    \]
\end{Theorem}

By contrast, it was established in \cite{RNW} that, with high probability, $a(Q_p) =\Omega(qp)$. Combining this observation with Theorem \ref{thm:random}, we see that a random set $Q_p$ with $p$ relatively large is very likely to have $a(Q_p)$ approximately equal to $qp$.

Other ranges of $p$ for this problem were also considered in \cite{RNW}, and a near-optimal bound for $p<1/q$ was proven. The problem of determining the behaviour of $a(Q_p)$ remains open in the range $q^{-1} \leq p \leq q^{-1/3}$.

\begin{proof}[Proof of Theorem \ref{thm:random}]
Write $m=pqf(q)$. Choose $\delta$ to be sufficiently small so that $a \geq -\frac{1}{3} + \delta$. Apply Lemma \ref{lem:cont2} with this choice of $\delta$ and with $\gamma=1$, to obtain a family of containers $\cC^2$. For $q$ sufficiently large, the probability that $Q_p$ contains an arc of size $m$ is at most
\[
|\cC^2| \binom{2q}{m} p^m.
\]
This is because an arc of size $m$ must be contained in some $C \in \mathcal C^2$, and each subset of size $m$ belongs to the random subset $Q_p$ with probability $p^m$. Applying the bounds from Lemma \ref{lem:cont2}, as well as the bound $\binom{a}{b} \leq \left (\frac{ea}{b} \right ) ^b$, it follows that
\begin{align*}
    \mathbb P [ a(Q_p) \geq m] & \leq |\cC^2| \binom{2q}{m} p^m
    \\& \leq  2^{q^{2/3+\delta}} \cdot  \left (\frac{2eqp}{m} \right ) ^m
    \\& =  2^{q^{2/3+\delta}} \cdot \left (\frac{2e}{f(q)} \right ) ^m
    \\& \leq  \left (\frac{4e}{f(q)} \right ) ^m.
\end{align*}
The last of these inequalities uses the fact that $p \geq q^{-\frac{1}{3}+\delta}$. Since $f(q)$ tends to infinity with $q$, it follows that
\[
\lim_{q \rightarrow \infty}  \mathbb P [ a(Q_p) \geq m] =0.
\]
\end{proof}

\section*{Acknowledgements}

The authors were supported by the Austrian Science Fund FWF Project P 34180. Part of this work was carried out during the Focused Research Workshop ``Testing Additive Structure", which was supported by the Heilbronn Institute for Mathematical Research. We are grateful to Cosmin Pohoata, Audie Warren and Adam Zsolt Wagner for helpful discussions.

\end{document}